\documentclass[12pt]{article}
\usepackage{latexsym, amsmath, amsfonts, amsthm, amssymb}
\usepackage{times}
\usepackage{a4wide}
\usepackage{graphicx, tocloft}
\usepackage{mathrsfs}
\usepackage{hyperref}

\def \Vol {{ \rm Vol }}

\def \RR {\mathbb R}

\def \EE {\mathbb E}

\def \ZZ {\mathbb Z}

\def \PP {\mathbb P}

\def \eps {\varepsilon}
\def \vphi {\varphi}

\def \cE {\mathcal E}
\def \cF {\mathcal F}

\def \cS {\mathcal S}

\DeclareMathOperator{\Tr}{Tr}

\newtheorem{theorem}{Theorem}[section]

\newtheorem{lemma}[theorem]{Lemma}

\newtheorem{proposition}[theorem]{Proposition}
\newtheorem{corollary}[theorem]{Corollary}

\theoremstyle{definition}
\newtheorem{remark}[theorem]{Remark}

\def\myffrac#1#2 in #3{\raise 2.6pt\hbox{$#3 #1$}\mkern-1.5mu\raise 0.8pt\hbox{$
		#3/$}\mkern-1.1mu\lower 1.5pt\hbox{$#3 #2$}}

\def\qed{\hfill $\vcenter{\hrule height .3mm
		\hbox {\vrule width .3mm height 2.1mm \kern 2mm \vrule width .3mm
			height 2.1mm} \hrule height .3mm}$ \bigskip}

\def \id {{\rm Id}}

\begin{document}

\title{Lattice packing of spheres in high dimensions \\ using a stochastically evolving ellipsoid}
\author{Boaz Klartag}
\date{}
\maketitle

\begin{abstract}
We prove that in any dimension $n$ there exists an origin-symmetric ellipsoid ${\mathcal{E}} \subset {\mathbb{R}}^n$ of volume $ c n^2 $ that contains
no points of ${\mathbb{Z}}^n$ other than the origin, where $c > 0$ is a universal constant. Equivalently, there exists a lattice sphere packing in ${\mathbb{R}}^n$ whose density is at least $cn^2 \cdot 2^{-n}$.
 Previously known constructions of sphere packings in ${\mathbb{R}}^n$ yielded densities of at most $C n \log n \cdot 2^{-n}$.
Our proof utilizes a stochastically evolving ellipsoid that accumulates at least  $c n^2$ lattice points on its boundary, while containing no lattice points
in its interior except for the origin.
\end{abstract}

    \section{Introduction}

    Let $n \geq 2$. A sphere packing in $\RR^n$ is a collection of disjoint  Euclidean balls of the same radius.
    A lattice in $\RR^n$  is the image of $\ZZ^n$ under an invertible, linear transformation $T: \RR^n \rightarrow \RR^n$.
    Thus, by a lattice in $\RR^n$ we always mean a lattice of full rank. The covolume of the lattice $L = T(\ZZ^n) \subset \RR^n$ is
     $$ \Vol_n(\RR^n / L) := |\det(T)|. $$
    A {\it lattice sphere packing} is a collection of disjoint  Euclidean balls, all of the same radius, whose centers
        form a lattice in $\RR^n$.
        The {\it density} of a lattice sphere packing is the proportion of space covered by the disjoint Euclidean balls
        of which it consists. Equivalently, if the lattice sphere packing consists of balls of radius $r$ whose centers form the lattice $L$,
        then its density  equals
        $$ \frac{\Vol_n(r B^n)}{\Vol_n(\RR^n / L)}, $$
        where $\Vol_n$ stands for $n$-dimensional volume in $\RR^n$, where $B^n \subset \RR^n$ is the open Euclidean ball of radius $1$ centered at the origin,
        and where $r A = \{ r x \, ; \, x \in A \}$ for $A \subset \RR^n$.   We write $\delta_n$ for the
        supremum of all densities of lattice sphere packings  in $\RR^n$. The  Minkowski-Hlawka theorem
        (see, e.g.,
        Gruber and Lekkerkerker \cite[Chapter 3]{GL}) implies that
        \begin{equation}  \delta_n \geq 2 \zeta(n) \cdot 2^{-n} = (2 + o(1)) \cdot 2^{-n}, \label{eq_1101} \end{equation}
        where $\zeta(n) = \sum_{k=1}^\infty k^{-n}$ and where $o(1)$ denotes a quantity that tends to zero as $n \to \infty$. The proof of the Minkowski-Hlawka theorem is probabilistic, and Problem X
        on Sarnak's list of open problems from the year 2000 asks for an explicit construction \cite{Sar}.
        The  bound (\ref{eq_1101}) was asymptotically improved in 1947 by Rogers \cite{rogers}, who showed that
        \begin{equation}
        \delta_n \geq c n \cdot 2^{-n} \label{eq_1659} \end{equation}
        for a universal constant $c > 0$. In his proof, Rogers used the Minkowski second theorem, as well as the concept of a random lattice and the Siegel summation
        formula, which we recall in Section \ref{sec5} below.

        \medskip The universal constant $c$ that Rogers' proof of (\ref{eq_1659}) yields satisfies $c \geq 2 / e$.
        This was subsequently improved by Davenport and Rogers \cite{DR}, who obtained (\ref{eq_1659}) with $c \approx 1.67$.  Ball \cite{B} used Bang's solution
        of Tarski's plank problem, and proved (\ref{eq_1659}) with $c = 2 - o(1)$. A plank is the region in space between two parallel hyperplanes,
        and the problem was to show that the sum of the widths of planks covering a convex body, is at least the body's minimal width.
        Vance \cite{vance} obtained $c \geq 6/e$ in dimensions divisible by $4$, by using
        random lattices with quaternionic symmetries. Her approach was further developed by Venkatesh \cite{venkatesh}, who used random lattices with sophisticated algebraic symmetries in order to show that
        \begin{equation}  \limsup_{n \rightarrow \infty} \frac{\delta_n}{n \cdot \log \log n \cdot 2^{-n}} \geq \frac{1}{2}. \label{eq_1202} \end{equation}
        Venkatesh \cite{venkatesh} also showed that $c \geq 2 \sinh^2(\pi e) / (\pi^2 e^3) - o(1) \approx 65963.8 - o(1)$.
        The constant~$1/2$ on the right-hand side of (\ref{eq_1202}) was improved to $1$ by Gargava and Viazovska~\cite{GV}.
        Campos, Jenssen, Michelen and Sahasrabudhe \cite{C} used graph-theoretic methods to prove the existence of a non-lattice sphere packing in $\RR^n$ of density $$ \left( \frac{1}{2} - o(1) \right) n \log n \cdot 2^{-n}. $$
        Graph theory was used earlier by
        Krivelevich, Litsyn and Vardy \cite{KLV} for the construction of a non-lattice sphere packing of density $c n \cdot 2^{-n}$ in $\RR^n$.
        Schmidt \cite{schmidt} proved (\ref{eq_1659}) by considering random lattices and by analyzing large hole events; these are rare events that occur with a probability of only $\exp(-\tilde{c} n)$.
        His analysis fits well with the theme that random lattices may sometimes be approximated by a Poisson process. The Poisson heuristic, which we recall below, was hinted at already in Rogers \cite{rogers3}.

        \medskip To summarize, up to logarithmic factors, several papers which are based on quite different ideas have  essentially arrived at  the same bound (\ref{eq_1659}) over the years.
        This bound has represented the state of the art on sphere packing in high dimensions -- again, up to logarithmic factors -- until now. We improve it as follows:

        \begin{theorem} For any $n \geq 2$,
        $$ \delta_n \geq c n^2 \cdot 2^{-n}, $$
        where $c > 0$ is a universal constant.
        \label{thm2}
        \end{theorem}
        The universal constant $c$ arising from our proof of Theorem \ref{thm2} can probably be computed numerically to a reasonable degree of accuracy; see Remark \ref{rem_951} below.
        Venkatesh \cite{venkatesh} conjectures that $2^{n} \delta_n$ grows at most polynomially in $n$.
        It is not entirely unlikely that Theorem \ref{thm2} is tight, up to a logarithmic correction. As for known upper bounds for $\delta_n$, in a short 1929 paper, Blichfeldt \cite{blich} proved that
        $$ \delta_n \leq \frac{n+2}{2} \cdot 2^{-n/2}. $$
        See also  Rankin \cite{rankin} and Rogers \cite{rogers58}. Kabatjanski\u{\i} and Leven\v{s}te\u{\i}n \cite{KL} improved the bound to roughly $\delta_n \lesssim (0.66)^n$,
        a result subsequently sharpened by constant factors by Cohn and Zhao \cite{CZ}
        and by Sardari and Zargar \cite{SZ}. These upper bounds also apply to non-lattice sphere packings.
        There is still a large gap between the known lower bound and the known upper bound
        for the optimal density of a sphere packing in high dimension. The precise optimal density is currently known in dimensions $2,3, 8$ and $24$,
        see Cohn \cite{C2} and references therein.

        \medskip
        By considering the lattice sphere packing $x + K/2 \ \ (x \in L)$,
        Theorem \ref{thm2} is easily seen to be equivalent to the following:

        \begin{theorem} Let $n \geq 2$ and let
        	$K \subset \RR^n$ be a Euclidean ball centered at the origin of volume
        \begin{equation}  \Vol_n(K) = c n^2. \label{eq_1533} \end{equation}
         Then there exists a lattice $L \subset \RR^n$ of covolume one with
        $ L \cap K = \{ 0 \}$. Here, $c > 0$ is a universal constant.
        \label{thm1}
        \end{theorem}

        An origin-symmetric ellipsoid in $\RR^n$ is the image of the unit ball $B^n$ under an invertible, linear map $T: \RR^n \rightarrow \RR^n$.
        Consider the lattice $L$ and the Euclidean ball $K$ from Theorem \ref{thm1}.
        Since $L$ may be represented as $L = T(\ZZ^n)$ for a linear map $T: \RR^n \rightarrow \RR^n$ with $|\det(T)| = 1$,
        we conclude from Theorem \ref{thm1} that the origin-symmetric ellipsoid $$ \cE = T^{-1}(K) \subset \RR^n $$ has volume $ c n^2, $ yet
        it contains no points from $\ZZ^n$ other than the origin. This implies the statement in the abstract of this paper.
        We conjecture that the conclusion of Theorem \ref{thm1} holds true for any origin-symmetric convex body $K \subset \RR^n$
        satisfying (\ref{eq_1533}), and not only for Euclidean balls and ellipsoids. See Schmidt \cite{schmidt1, schmidt2}
        for a proof under the weaker assumption that $\Vol_n(K) \leq c n$.

        \medskip Before presenting the main ideas of the proof of Theorem \ref{thm1},
        let us briefly discuss the proof of (\ref{eq_1659}) from Rogers \cite{rogers}.
        Consider a random lattice $L \subset \RR^n$ satisfying $\Vol_n(\RR^n / L) = \Vol_n(B^n)$. By using the Siegel summation formula, it is shown
         that with positive probability,
        $$ \prod_{i=1}^n \lambda_i \geq c n $$
        where $0 < \lambda_1 \leq \lambda_2 \leq \ldots \leq \lambda_n$ are the successive minima of the lattice $L$. Minkowski's second theorem
        is then used in order to find a linear map $T: \RR^n \rightarrow \RR^n$ with $|\det(T)| \geq \prod_i \lambda_i$ such that $T(B^n) \cap L = \{ 0 \}$.
        Intuitively, the ellipsoid $T(B^n)$ constructed this way only ``interacts'' with $n$ vectors from the lattice -- the ones corresponding to the successive minima.

    \medskip In contrast, an ellipsoid in $\RR^n$ is determined by $n(n+1)/2$ parameters, and it is reasonable to expect it
    to ``interact'' with roughly $n^2$ lattice points. In fact, it is not too difficult
    to show that there exists an origin-symmetric ellipsoid $\cE \subset \RR^n$ with $\cE \cap \ZZ^n = \{ 0 \}$ such that
    \begin{equation}   |\partial \cE \cap \ZZ^n| \geq n(n+1). \label{eq_1815} \end{equation}
    Here, $| A |$ is the cardinality of the set $A \subset \RR^n$, and $\partial \cE$ is the boundary of the ellipsoid $\cE$.
    See Theorem VIII in Cassels \cite[Section V.8]{cassels} or Remark \ref{rem_1814} below for a proof of (\ref{eq_1815}).

    \medskip Our construction of the ellipsoid $\cE \subset \RR^n$ begins
    with a random lattice $L \subset \RR^n$ satisfying $\Vol_n(\RR^n / L) = \Vol_n(B^n)$.
    Consider a relatively large Euclidean ball
    disjoint from $L \setminus \{0 \}$, and let us run a Brownian-type stochastic motion in the space of ellipsoids, starting from this Euclidean ball.
    The crucial property of our stochastic process is that whenever the  evolving ellipsoid $$ \cE_t = \{ x \in \RR^n \, ; \, A_t x \cdot x < 1 \} $$ hits a non-zero lattice point, it keeps it on its boundary at all later times. In other words, if the ellipsoid hits the point $0 \neq x_0 \in L$ at time $t_0$, then we make sure
    that for $t > t_0$,
    \begin{equation}  A_t x_0 \cdot x_0 = 1. \label{eq_1207} \end{equation}
    Equation (\ref{eq_1207}) imposes a one-dimensional linear constraint on the matrix $A_t$, and  the stochastic evolution of $A_t$ may be continued in the linear subspace
    of matrices obeying this constraint.
    The vector  space of all real symmetric $n \times n$ matrices, denoted by $$ \RR^{n \times n}_{\mathrm{sym}}, $$
    has dimension $n(n+1)/2$. Hence our evolving ellipsoid freezes only when it has accumulated $n(n+1)$ lattice points on its boundary; note that
    these contact points come in pairs: $x_0 \in L$ and $-x_0 \in L$. We remark that similar ideas were used in different contexts in constructions described
    in Lovett and Meka \cite{LM} as well as in \cite{complex_waist}.

    \medskip
    Intuitively, the random lattice $L$ behaves somewhat like a Poisson process in $\RR^n$ of intensity $$ 1/\Vol_n(B^n). $$ 
    Thus, one might expect the ellipsoid to cover a total volume of about $c n^2 \cdot \Vol_n(B^n)$ during its evolution, since it manages to find $n(n+1)$ lattice points.
    Our evolving ellipsoid expands and contracts in a random fashion, and its volume is not monotone. Nevertheless,
    we hope to define an appropriate stochastic motion such that the evolving ellipsoid does
    not recede significantly from regions near the accumulated contact points.
    In this case, the evolving ellipsoid is expected to reach a volume of $c n^2 \cdot \Vol_n(B^n)$ while remaining $L$-free.

   \medskip In our construction, we let the matrix $A_t$ evolve as a Brownian motion in the space of symmetric matrices (``Dyson Brownian motion''),
subject to linear constraints of the form (\ref{eq_1207}). These constraints also ensure that the process never leaves the cone of symmetric, positive-definite matrices.
The volume of the ellipsoid $\cE_t$ is determined by $\det A_t$. Fortunately, $\log\det$ is a concave function on the cone of positive-definite matrices,
and since $A_t$ is a martingale, the function $$ t \mapsto \EE \log \det A_t $$ is non-increasing.
In fact, the change in the quadratic variation of $A_t$ is determined by the linear subspace of matrices satisfying the constraints at any given time, and hence
the time derivative of $\EE \log \det A_t$ is related to the number of contact points. By analyzing the rate at which the ellipsoid $\cE_t$
accumulates lattice points, we may therefore control the growth rate of $\EE \log \Vol_n(\cE_t)$. This enables us to show that at a certain fixed time $T>0$,
with positive probability, the ellipsoid $\cE_T$ has a large volume while containing no lattice points in its interior.

    \medskip In the remainder of this paper we transform these vague heuristics into
    a mathematical proof. In Section \ref{sec2} we construct the stochastically evolving
    ellipsoid for a given lattice (or a lattice-like set). In Section \ref{sec3} we study the volume growth of the evolving ellipsoid,
    and in Section \ref{sec4} we analyze the  rate at which it accumulates lattice points.
    In Section \ref{sec5} we discuss random lattices, and complete the proof of Theorem \ref{thm1}.

    \medskip The linear space $\RR^{n \times n}_{\mathrm{sym}}$ is a Euclidean space equipped with the scalar product
    $$ \langle A, B \rangle = \Tr[AB] \qquad \qquad \qquad (A,B \in \RR^{n \times n}_{\mathrm{sym}}), $$
    where $\Tr[A]$ is the trace of the matrix $A \in \RR^{n \times n}$.
    We denote the collection of positive-definite, symmetric $n \times n$ matrices by
    $$ \RR^{n \times n}_+ \subset \RR^{n \times n}_{\mathrm{sym}}. $$
     We write that $A \geq B$ (respectively, $A > B$) for two matrices
    $A, B \in \RR^{n \times n}_{\mathrm{sym}}$ if $A - B$ is positive semi-definite (respectively, positive-definite). We write $\id$ for the identity matrix.
    The Euclidean norm of $x = (x_1,\ldots,x_n) \in \RR^n$ is denoted by $|x| = \sqrt{\sum_i x_i^2}$.
    For $x, y \in \RR^n$ we write $x \cdot y = \sum_{i=1}^n x_i y_i$ for their standard
    scalar product, and $x \otimes y = (x_i y_j)_{i,j=1,\ldots,n} \in \RR^{n \times n}$ for their tensor product.
    The natural logarithm is denoted by $\log$. A subset $A \subset \RR^n$ is origin-symmetric if $A = -A$.
    All ellipsoids are assumed to be open and origin-symmetric. A random variable $X$ is centered when $\EE X = 0$.

    \medskip Throughout this paper, we write $c, C, \tilde{C}, c', \hat{C}, \bar{C}$ etc. for various positive universal constants
    whose value may change from one line to the next. We write $C_0, C_1, c_0$ etc. -- that is, the letters $C$ or $c$ with numerical subscripts --
    for positive universal constants that remain fixed throughout the paper. In proving Theorem \ref{thm1}, we may assume that the dimension $n$ is sufficiently large;
    this is our standing assumption throughout the text.

    \medskip
    \emph{Acknowledgement}. I am grateful to Barak Weiss for interesting discussions and for his encouragement.
    Supported by a grant from the Israel Science Foundation (ISF).

\section{Constructing a stochastically evolving  ellipsoid}
\label{sec2}

Let $L \subset \RR^n$ be a discrete subset of $\RR^n$ such that
for any  origin-symmetric ellipsoid $\cE \subset \RR^n$ with $\cE \cap L \subseteq \{ 0 \}$, we have
\begin{equation}  \Vol_n(\cE) \leq C_L \qquad \text{and} \qquad |\partial \cE \cap L| \leq \tilde{C}_L, \label{eq_1358} \end{equation}
for some constants $C_L, \tilde{C}_L > 0$ depending only on $L$. We refer to such a discrete set $L$ as a {\it lattice-like set}.
The most important case is when $L \subset \RR^n$ is a lattice; in this case the inequalities in (\ref{eq_1358}) hold true with $C_L = 2^n \cdot \Vol_n(\RR^n / L)$,
by Minkowski's first theorem, and with
\begin{equation}  \tilde{C}_L = 2 \cdot (2^n - 1) \label{eq_643} \end{equation} by an elementary argument which we reproduce in the Appendix below.
For a symmetric matrix $A \in \RR^{n \times n}_{\mathrm{sym}}$ we consider the open set
\begin{equation}  \cE_A = \left \{ x \in \RR^n \, ; \, A x \cdot x < 1 \right \}. \label{eq_1148} \end{equation}
Its boundary $\partial \cE_A$ is the collection of all $x \in \RR^n$ with $A x \cdot x = 1$. The matrix $A \in \RR^{ n \times n}_{\mathrm{sym}}$
is positive-definite if and only if the set $\cE_A$ is an  ellipsoid, in which case
\begin{equation}  \Vol_n(\cE_A) = \det(A)^{-1/2} \cdot \Vol_n(B^n).
	\label{eq_1128} \end{equation}
When $A$ is  not positive-definite, necessarily $\Vol_n(\cE_A) = \infty$.
We say that an open subset  $\cE \subseteq \RR^n$ is {\it  $L$-free} if $\cE \cap L \subseteq \{ 0 \}$.
When we write that the matrix $A \in \RR^{n \times n}_{\mathrm{sym}}$ is $L$-free, we mean that the open set $\cE_A$ is $L$-free. It follows from (\ref{eq_1358}) that the volume
of an $L$-free ellipsoid is at most $C_L$, and that it contains at most $\tilde{C}_L$ points on its boundary.

\medskip A point belonging both to the boundary $\partial \cE_A$ and to the discrete set $L$ is referred to as a contact point. The following lemma describes a continuous deformation of an $L$-free ellipsoid that keeps all of its contact points. We defer the standard proof to the Appendix.

\begin{lemma} Let $M_t \in \RR^{n \times n}_{\mathrm{sym}} \ (t \geq 0)$ be a family of matrices depending continuously on $t \geq 0$, such that not all of the matrices are positive-definite. 	Assume that the matrix $M_0 \in \RR^{n \times n}_{\mathrm{sym}}$ is positive-definite and $L$-free, and that for  all $t \geq 0$,
\begin{equation}
	\partial \cE_{M_0} \cap L \subseteq \partial \cE_{M_t} \cap L. \label{eq_1935} \end{equation}
Then the following hold:
\begin{enumerate}
	\item[(A)] Denote $$ \tau := \sup \left \{ \, t \geq 0 \, ; \, M_s \ \textrm{is} \ L\textrm{-free with } \partial \cE_{M_s} \cap L = \partial \cE_{M_0} \cap L \textrm{ for all } s \in [0,t] \, \right \}. $$
Then $0 < \tau < \infty$.
	\item[(B)] The symmetric matrix $M_t$ is positive-definite and $L$-free for all $0 \leq t \leq \tau$.
	\item[(C)] We gained at least one additional contact point at time $\tau$. That is,
\begin{equation}  \partial \cE_{M_0} \cap L \subsetneq \partial \cE_{M_\tau} \cap L. \label{eq_1608} \end{equation}
\end{enumerate}
\label{lem_2104}
\end{lemma}

We recall that the {\it standard Brownian motion} in a finite-dimensional, real, inner product space~$V$ is a centered, continuous, Gaussian process $(W_t)_{t \geq 0}$ attaining values in $V$, with $W_0 = 0$, and with independent increments\footnote{i.e., $W_t - W_s$ is independent of $W_s - W_r$ for all $0 \leq r < s < t$.}, such that for all $t > s \geq 0$ and a linear functional $f: V \rightarrow \RR$,
$$ \EE |f(W_t - W_s)|^2 = (t -s) \|f\|^2. $$
Here, $\|f\| = \sup_{0 \neq v \in V} |f(v)| / \|v\|$ and $\| v \| = \sqrt{\langle v, v \rangle}$. We refer the reader e.g. to Le Gall~\cite{legall}, {\O}ksendal \cite{O} or Revuz and Yor \cite{RY} for background on
Brownian motion and stochastic analysis.

\medskip
The {\it Dyson Brownian motion} is a standard Brownian motion $(W_t)_{t \geq 0}$ in the Euclidean space $\RR^{n \times n}_{\mathrm{sym}}$.
 For $A \in \RR^{n \times n}_{\mathrm{sym}}$  consider the subspace
\begin{equation}  F_A = \left \{ B \in \RR^{n \times n}_{\mathrm{sym}} \, ; \, \forall x \in \partial \cE_A \cap L, \ B x \cdot x = 0 \right \}, \label{eq_1730} \end{equation}
where $\cE_A$ is defined in (\ref{eq_1148}). We write $\pi_A: \RR^{n \times n}_{\mathrm{sym}} \to \RR^{n \times n}_{\mathrm{sym}}$ for the orthogonal projection operator
onto the subspace $F_A$.
The following lemma explains how to randomly evolve an $L$-free ellipsoid until we gain an additional contact point.

\begin{lemma} Let $M_0 \in \RR^{n \times n}_+$ be an $L$-free matrix with $ F_{M_0} \neq \{ 0 \}$.
Let $(W_t)_{t \geq 0}$ be a Dyson Brownian motion in $\RR^{n \times n}_{\mathrm{sym}}$. For $t \geq 0$
denote
\begin{equation}  M_t = M_0 + \pi_{M_0}(W_t). \label{eq_1208} \end{equation}
Then, with probability one, the random variable $$ \tau: = \sup \{ \, t \geq 0 \, ; \, M_s \ \textrm{is} \ L\textrm{-free with }
\partial \cE_{M_s} \cap L = \partial \cE_{M_0} \cap L \textrm{ for all } s \in [0,t] \, \}, $$
is non-zero and finite. Moreover, almost surely, for $0 \leq t \leq \tau$ the set $\cE_{M_t}$ is an $L$-free ellipsoid, and
\begin{equation} \partial \cE_{M_0} \cap L \subsetneq \partial \cE_{M_\tau} \cap L.
\label{eq_1740} \end{equation}
\label{lem_1800}
\end{lemma}

\begin{proof} Since $ F_{M_0} \neq \{ 0 \}$, the linear projection $\pi_{M_0}: \RR^{n \times n}_{\mathrm{sym}} \rightarrow \RR^{n \times n}_{\mathrm{sym}}$ is not identically zero. Hence there exists $x_0 \in \RR^n$ such that $\pi_{M_0}(x_0 \otimes x_0) \neq 0$. Almost surely, a Brownian motion in $\RR$ does not remain bounded from below indefinitely. Therefore, almost surely
\begin{equation}  \liminf_{t \rightarrow \infty} \pi_{M_0}(W_t) x_0 \cdot x_0 =
 \liminf_{t \rightarrow \infty} \langle W_t, \pi_{M_0}(x_0 \otimes x_0) \rangle =
  -\infty. \label{eq_1750} \end{equation}
 It follows from (\ref{eq_1208}) and (\ref{eq_1750}) that almost surely, $(M_t)_{t \geq 0}$ is {\it not} a family of positive-definite matrices.
In order to verify all of the other assumptions of Lemma \ref{lem_2104}, we note that if $x \in \partial \cE_{M_0} \cap L$
then by (\ref{eq_1730}),
\begin{equation}  B x \cdot x = 0 \qquad \qquad \qquad \text{for all} \ B \in F_{M_0}. \label{eq_1659_} \end{equation}
Recall that $\pi_{M_0}(W_t) \in F_{M_0}$. Thus, by (\ref{eq_1659_}), for all $t \geq 0$ and $x \in \partial \cE_{M_0} \cap L$,
$$ M_t x \cdot x = M_0 x \cdot x + \pi_{M_0}(W_t) x \cdot x = M_0 x \cdot x = 1. $$
Hence $x \in \partial \cE_{M_t} \cap L$ for all $t \geq 0$. We have thus shown that almost surely, for all $t \geq 0$,
$$
\partial \cE_{M_0} \cap L \subseteq \partial \cE_{M_t} \cap L.
$$
We have verified all of the assumptions of Lemma \ref{lem_2104}. We may therefore apply the lemma, and conclude that almost surely the random variable  $\tau$
is finite and non-zero. From conclusion (B) of Lemma \ref{lem_2104} we learn that almost surely, for all $0 \leq t \leq \tau$
the set $\cE_{M_t}$ is an $L$-free ellipsoid. Conclusion (C) of Lemma \ref{lem_2104} implies (\ref{eq_1740}).
\end{proof}

Recall that the (completed) filtration associated with the Brownian motion $(W_t)_{t \geq 0}$
is $(\cF_t)_{t \geq 0}$, where $\cF_t$ is the $\sigma$-algebra generated by the random variables $(W_s)_{0 \leq s \leq t}$,
as well as by all subsets of all measurable null sets, see e.g. Le Gall \cite[Chapter 3]{legall}.
A stochastic process $(A_t)_{t \geq 0}$ is {\it adapted} to this filtration if for any fixed $t \geq 0$, the random variable $A_t$ is measurable with respect to~$\cF_t$.
It is a martingale if $\EE[ A_t | \cF_s ] = A_s$ for all $t \geq s \geq 0$.
A stopping time $\tau$ is a random variable attaining values in $[0, \infty)$ such that for any fixed $t \geq 0$, the event $\{ \tau \leq t \}$ is measurable with respect to~$\cF_t$.
For example, the random variable $\tau$ from Lemma \ref{lem_1800} is a stopping time.
The following proposition describes our construction of the stochastically evolving ellipsoid associated with the lattice-like set $L \subset \RR^n$.

\begin{proposition} Let $a_0 > 0$ be such that the matrix $a_0 \cdot \id \in \RR^{n \times n}$ is $L$-free.
Let $(W_t)_{t \geq 0}$ be a Dyson Brownian motion in $\RR^{n \times n}_{\mathrm{sym}}$. Then there exists a continuous stochastic process $(A_t)_{t \geq 0}$, attaining values in $\RR^{n \times n}_{\mathrm{sym}}$ and adapted to the filtration associated with $(W_t)_{t \geq 0}$, with the following properties:
\begin{enumerate}
	\item[(A)] Abbreviate $\pi_t = \pi_{A_t}$. Then there exist a bounded, integer-valued random variable $M \geq 0$ and stopping times $0 = \tau_0 < \tau_1 < \tau_2 < \ldots$
for which the following hold: for any fixed $i \geq 1$ and $t > 0$, if $i \leq M$ and $t \in [\tau_{i-1}, \tau_i)$ then $\pi_t = \pi_{\tau_{i-1}}$ and
	\begin{equation}
		A_t = A_{\tau_{i-1}} + \pi_{t} \left( W_t - W_{\tau_{i-1}} \right). \label{eq_1803}
	\end{equation}
\item[(B)] For $t \geq \tau_M$ we have  $A_t = A_{\tau_M}$ and $\pi_t = 0$. Moreover, $A_0 = a_0 \cdot \id$.
	\item[(C)] Almost surely, for all $t \geq 0$ the matrix $A_t$ is positive-definite and $L$-free.
\item[(D)] Set $\cE_t := \cE_{A_t}$. Then almost surely, $\partial \cE_s \cap L \subseteq \partial \cE_t \cap L$ for all $0 \leq s \leq t$.
\item[(E)] Denote $F_t := F_{A_t}$. Then almost surely,
\begin{equation}  \tau_* := \inf \{ t \geq 0 \, ; \, F_t = \{ 0 \} \}
\label{eq_1507} \end{equation}
	is finite, and $\tau_* = \tau_M$.
\end{enumerate}
\label{prop_1510}
\end{proposition}

\begin{proof} We will recursively iterate the construction of Lemma \ref{lem_1800}.
Set
$$
	A_0 = a_0 \cdot \id,
$$ 
 and $\tau_0 = 0$. We will inductively construct
stopping times
\begin{equation} 0 = \tau_0 < \tau_1 < \tau_2 < \ldots  \label{eq_1339} \end{equation}
and symmetric matrices $(A_{\tau_i})_{i \geq 1}$  such that almost surely, the random variable $\tau_i$ is finite and the matrix $A_{\tau_i}$ is a positive-definite,  $L$-free matrix for all~$i$. For the base of the induction, we note that the matrix $A_{\tau_0}$ is positive-definite and $L$-free, by assumption.

\medskip Let $i \geq 1$ and suppose that $\tau_{i-1}$ and $A_{\tau_{i-1}}$ have been constructed such that almost surely $\tau_{i-1}$ is finite, and $A_{\tau_{i-1}}$ is positive-definite and $L$-free. Let us construct $\tau_i$ and $A_{\tau_i}$.
 If
\begin{equation}   F_{{\tau_{i-1}}} = \{ 0 \} \label{eq_1707} \end{equation}
then we simply set $A_t := A_{\tau_{i-1}}$ for $t > \tau_{i-1}$. We also define $M := i-1$ and $\tau_{M+j} := \tau_M + j$ for $j \geq 1$, and end the recursive construction. By the induction hypothesis, $A_{\tau_j}$ is a positive-definite, $L$-free matrix for all $j \geq i$.
This completes the description of the recursion step in the case where (\ref{eq_1707}) holds true.
Suppose now that
\begin{equation}  F_{{\tau_{i-1}}} \neq \{ 0 \}.
\label{eq_1329} \end{equation}
Define
\begin{equation} W_t^{(i)} := W_{t + \tau_{i-1}} - W_{\tau_{i-1}} \qquad \qquad (t \geq 0), \label{eq_1453}
\end{equation}
which is a standard Brownian motion in $\RR^{n \times n}_{\mathrm{sym}}$, or in other words, a Dyson Brownian motion.
Set \begin{equation} M_0 = A_{\tau_{i-1}}, \label{eq_1830} \end{equation}
which is positive-definite and $L$-free by the induction hypothesis.
We know that $F_{M_0} \neq \{ 0 \}$, thanks to (\ref{eq_1329}). Denote
\begin{equation} M_t = M_0 + \pi_{M_0}(W_t^{(i)}), \label{eq_1829} \end{equation}
and apply Lemma \ref{lem_1800}. From the conclusion of the lemma, almost surely the stopping time
\begin{equation}  \tau_i := \tau_{i-1} + \sup \{ t \geq 0 \, ; \, M_s \ \textrm{is} \ L\textrm{-free with } \partial \cE_{M_s} \cap L = \partial \cE_{M_0} \cap L \textrm{ for all } s \in [0,t] \}, \label{eq_1436} \end{equation}
is finite with $\tau_i > \tau_{i-1}$. Moreover, almost surely $\cE_{M_t}$ is an $L$-free ellipsoid for $0 \leq t \leq \tau_i - \tau_{i-1}$.  Therefore, setting
\begin{equation}  A_{t} := M_{t - \tau_{i-1}} \qquad \qquad \qquad \text{for} \ t \in (\tau_{i-1}, \tau_i] \label{eq_1338} \end{equation}
we see that $A_t$ is positive-definite and $L$-free for $t \in [\tau_{i-1}, \tau_i]$. Almost surely, the matrix $A_t$ depends continuously on $t \in [\tau_{i-1}, \tau_i]$.
Furthermore, from conclusion (\ref{eq_1740}) of Lemma \ref{lem_1800} we learn that almost surely,
\begin{equation} \partial \cE_{A_{\tau_{i-1}}} \cap L \subsetneq \partial \cE_{A_{\tau_{i}}} \cap L.
\label{eq_1333} \end{equation}
This completes the description of the recursive construction of $\tau_i$ and $A_{\tau_i}$ for all $i \geq 0$. It follows from (\ref{eq_1339}) that along the way we defined the random matrix $A_t$ for all $0 < t \leq \tau_M$, via formula (\ref{eq_1338}). For completeness, set
\begin{equation} A_t = A_{\tau_M} \qquad  \textrm{for}  \ t > \tau_M.
\label{eq_1227_} \end{equation}	
	 Thus, almost surely the stochastic process $(A_t)_{t \geq 0}$ is well-defined and continuous. Let us discuss the basic properties of this construction.

\medskip We first claim that the random variable $M$, which is the number of steps in the construction, is a bounded random variable. Indeed,
relation (\ref{eq_1333}) holds true for all $i=1,\ldots,M$. Therefore,
\begin{equation}  |\partial \cE_{A_{\tau_{i}}} \cap L| \geq i \qquad \qquad \qquad (i=1,\ldots,M), \label{eq_1337} \end{equation}
while the ellipsoid $\cE_{A_{\tau_{i}}}$ is $L$-free. It follows from (\ref{eq_1358}) and (\ref{eq_1337}) that  almost surely $M \leq \tilde{C}_L$, and hence $M$ is a bounded random variable.
We conclude that the random variable $\tau_M$ is almost surely finite, being almost surely the sum of finitely many numbers.

\medskip Next, by the construction of $A_t$ in (\ref{eq_1338}), the matrix $A_t$ is almost surely positive-definite and $L$-free for $t \in [\tau_{i-1}, \tau_i]$ for all $i=1,\ldots,M$.
It thus follows from  (\ref{eq_1339}) that $A_t$ is positive-definite and $L$-free for $t \in [0, \tau_M]$. Observe that by (\ref{eq_1436}) and (\ref{eq_1338}),
for any $i=1,\ldots,M$ and $t \in [\tau_{i-1}, \tau_{i})$,
\begin{equation}  \partial \cE_{A_t} \cap L = \partial \cE_{A_{\tau_{i-1}}} \cap L. \label{eq_1457} \end{equation}
 From  (\ref{eq_1453}), (\ref{eq_1830}), (\ref{eq_1829}), (\ref{eq_1338}) and (\ref{eq_1457}), for  $i=1,\ldots,M$ and $t \in [\tau_{i-1}, \tau_{i})$
 we have $F_{t} = F_{{\tau_{i-1}}}$ and
$$ A_t = A_{\tau_{i-1}} + \pi_{A_{\tau_{i-1}}}(W_{t - \tau_{i-1}}^{(i)})
= A_{\tau_{i-1}} + \pi_{A_t}(W_t - W_{\tau_{i-1}}).
$$
Since $A_t$ depends continuously on $t$, and since $A_t$ is constant for $t \in [\tau_M, \infty)$ by (\ref{eq_1227_}), conclusion (A) and conclusion (B) are proven.
Note that the matrix $A_t$ is determined by $a_0, L$ and $(W_s)_{0 \leq s \leq t}$. In particular,
the stochastic process $(A_t)_{t \geq 0}$ is adapted to the filtration associated with the Dyson Brownian motion.

\medskip Conclusion (C) holds true as $A_t$ is positive-definite and $L$-free for $t \in [0, \tau_M]$, and $A_t = A_{\tau_M}$  for $t \in [\tau_M, \infty)$.
Conclusion (D) holds true in view of (\ref{eq_1333})  and (\ref{eq_1457}).

\medskip From our construction, if $\PP( M = 0) > 0$ then $F_{A_0} = \{ 0 \}$. In this case, actually $M \equiv 0$ almost surely and conclusion~(E)
trivially holds  with $\tau_* = \tau_M = 0$, where $\tau_*$ is defined in (\ref{eq_1507}). Otherwise, $M \geq 1$ almost surely. In this case, the subspace $F_t = F_{A_t}$ is constant and different from $\{ 0 \}$ for $t \in [\tau_{i-1}, \tau_i)$ and $i=1,\ldots,M$.
We always have $F_{\tau_M} = \{ 0 \}$. It thus follows that $\tau_* = \tau_M$. Thus the stopping time $\tau_*$ is almost surely finite, completing the proof of (E).
\end{proof}

We refer to the stochastic process $(\cE_{t})_{t \geq 0}$ from Proposition \ref{prop_1510} as the {\it stochastically evolving ellipsoid}.
The volume of the $L$-free ellipsoid $\cE_t$ may increase or decrease with $t$, but it remains bounded at all times. In fact, it  follows from (\ref{eq_1358}), (\ref{eq_1128}) and Proposition \ref{prop_1510}(C)
that almost surely,
\begin{equation}  \det A_t \geq c_L  \qquad \qquad \qquad \text{for all} \ t \geq 0, \label{eq_1238} \end{equation}
with $c_L = (C_L / \Vol_n(B^n))^{-2}$. The It\^o integral interpretation of conclusions (A) and (B) of
Proposition \ref{prop_1510} is given in the following:

\begin{corollary} Under the notation and assumptions of Proposition \ref{prop_1510},
for all $t \geq 0$,
\begin{equation}  A_t = a_0 \cdot \id  +  \int_0^t  \pi_s (d W_s). \label{eq_1701_} \end{equation}
Thus $A_0 = a_0 \cdot \id$ and we have the stochastic differential equation
\begin{equation}  d A_t =  \pi_t(dW_t).
	\label{eq_1701} \end{equation}	
\label{cor_1852}
\end{corollary}

\begin{proof} The It\^o integral on the right-hand side of (\ref{eq_1701_}) may be defined as
\begin{equation}  \int_0^t  \pi_s (d W_s) = \lim_{\eps(P) \rightarrow 0} \sum_{i=1}^{N_P} \pi_{t_{i-1}} \left( W_{t_i} - W_{t_{i-1}} \right), \label{eq_1801}
\end{equation}
where $P = \{ 0 = t_0 < t_1 < \ldots < t_{N_P} = t \}$ is a non-random partition of $[0,t]$ into $N_P$ intervals
and $\eps(P) = \max_{1 \leq i \leq N_P} |t_{i} - t_{i-1}|$. For a fixed $t > 0$, the convergence
of the $\RR^{n \times n}_{\mathrm{sym}}$-valued random variables in (\ref{eq_1801}) is in the sense of $L^2$.
See e.g. \cite[Chapter 3]{O} for explanations on the construction of the It\^o integral using $L^2$-convergence,
as well as for the fact that by modifying the collection of random variables $$ \left( \int_0^t  \pi_s (d W_s) \right)_{t \geq 0} $$
on a null set, we may assume that the map $t \to \int_0^t  \pi_s (d W_s)$ is almost-surely continuous in $[0, \infty)$.
In order to prove (\ref{eq_1701_}), we use conclusions (A) and (B) of Proposition \ref{prop_1510}.
By these conclusions, if $Q = \{ 0 = s_0 < s_1 < \ldots < s_{N_Q} = t \}$
is a {\it random} partition of $[0,t]$ into $N_Q$ intervals such that $\tau_i \in Q$ for all $i=1,\ldots,M$ with $\tau_i \leq t$, then,
\begin{equation}  A_t = a_0 \cdot \id + \sum_{i=1}^{N_Q} \pi_{s_{i-1}} \left( W_{s_i} - W_{s_{i-1}} \right). \label{eq_1136} \end{equation}
Now (\ref{eq_1701_}) follows from (\ref{eq_1801}) and (\ref{eq_1136}) by approximating $P$ with a refined partition that
contains all of the $\tau_i$'s that lie in the interval $[0,t]$. Indeed, since $M$ is a bounded random variable, the approximation error can be made arbitrarily small in~$L^2$.
Finally, equation (\ref{eq_1701}) is simply the stochastic differential equation rewriting of~(\ref{eq_1701_}).
\end{proof}

\section{The shape and volume of the evolving ellipsoid}
\label{sec3}

 Let $L \subset \RR^n$ be a lattice. Assume that $a_0 > 0$ is such that the matrix
$$ a_0 \cdot \id \in \RR^{n \times n} $$
is $L$-free. Let $(A_t)_{t \geq 0}$ be the stochastic process constructed in Proposition \ref{prop_1510}, driven
by the Dyson Brownian motion $(W_t)_{t \geq 0}$ in $\RR^{n \times n}_{\mathrm{sym}}$.

\begin{lemma} There exist two Dyson Brownian motions $(W_t^{(1)})_{t \geq 0}$ and $(W_t^{(2)})_{t \geq 0}$ in $\RR^{n \times n}_{\mathrm{sym}}$ such that for all $t \geq 0$,
\begin{equation}  A_t = a_0 \cdot \id + \frac{W_t^{(1)} + W_t^{(2)}}{2}.
	\label{eq_1046} \end{equation}
	\label{lem_1033}
\end{lemma}

\begin{proof} We use an idea that is attributed to
	Bernard Maurey, see Eldan and Lehec \cite[Proposition 4]{EL}.	
	Recall from Proposition \ref{prop_1510}  the linear map $$ \pi_t: \RR^{n \times n}_{\mathrm{sym}} \to \RR^{n \times n}_{\mathrm{sym}}, $$
which is the orthogonal projection operator onto the subspace of all matrices $B \in \RR^{n \times n}_{\mathrm{sym}}$ satisfying $B x \cdot x = 0$ for all $x \in \partial \cE_{t} \cap L$.
Almost surely, for all $t \geq 0$ the linear map $\pi_t$  is a symmetric operator  and
$$ 0 \leq \pi_t \leq \id, $$
in the sense of symmetric operators on the Euclidean space $\RR^{n \times n}_{\mathrm{sym}}$. It follows from Corollary \ref{cor_1852} that $(A_t)_{t \geq 0}$ is a martingale in $\RR^{n \times n}_{\mathrm{sym}}$. Its quadratic variation process is
\begin{equation}
[A]_t = \int_0^t \pi_s^2 ds = \int_0^t \pi_s ds. \label{eq_1023}
\qquad \qquad \qquad (t > 0).
\end{equation}
Denote $\tilde{\pi}_t = \id - \pi_t: \RR^{n \times n}_{\mathrm{sym}} \to \RR^{n \times n}_{\mathrm{sym}}$ and set
$$ W_t^{(1)} = \int_0^t \pi_s(d W_s) + \int_0^t \tilde{\pi}_s(d W_s) $$
	and
$$  W_t^{(2)} = \int_0^t \pi_s(d W_s)  - \int_0^t \tilde{\pi}_s(d W_s). $$
Thus $(W_t^{(1)})_{t \geq 0}$ and $(W_t^{(2)})_{t \geq 0}$ are well-defined, continuous martingales in $\RR^{n \times n}_{\mathrm{sym}}$, with
$$W_t^{(1)}  + W_t^{(2)}  = 2 \int_0^t \pi_s(d W_s) = 2 (A_t - a_0 \cdot \id), $$
where the last passage follows from  Corollary \ref{cor_1852}. This proves
the desired conclusion (\ref{eq_1046}). Recall that Paul L\'evy characterized standard Brownian
motion as the unique continuous martingale starting from zero whose quadratic variation process
is almost surely $t \cdot \id$ for all $t > 0$. See e.g. \cite[Theorem 5.12]{legall}.
The quadratic variation processes of our two martingales satisfy, for $t > 0$,
$$ [W^{(1)}]_t = \int_0^t (\pi_s + \tilde{\pi}_s)^2 ds = t \cdot \id $$
and
$$ [W^{(2)}]_t = \int_0^t (\pi_s - \tilde{\pi}_s)^2 ds = t \cdot \id. $$
Note that $ W_0^{(1)} = W_0^{(2)} = 0$. Thus, by L\'evy's characterization,
both $(W_t^{(1)})_{t \geq 0}$ and $(W_t^{(2)})_{t \geq 0}$ are standard Brownian motions in $\RR^{n \times n}_{\mathrm{sym}}$. In other words, both stochastic processes are Dyson Brownian motions. \end{proof}

Write $\| A \|_{op} = \sup_{0 \neq x \in \RR^n} |Ax| / |x|$ for the operator norm of the matrix $A \in \RR^{n \times n}$.

\begin{corollary} For any $t > 0$ and $r \geq \sqrt{tn}$,
	$$ \PP  \left( \| A_t - a_0 \cdot \id \|_{op} \geq C_0 r \right) \leq C \exp(- r^2 / t), $$
	where $C, C_0 > 0$ are universal constants. \label{cor_1833}
\end{corollary}

\begin{proof} From Lemma \ref{lem_1033}, for any $r > 0$,
\begin{align*}  \PP \left( \| A_t - a_0 \cdot \id \|_{op}  \geq r \right)
 & = \PP \left( \left \| \frac{W_t^{(1)} + W_t^{(2)}}{2}  \right \|_{op} \geq r \right)
 \\ & \leq
  2 \PP \left( \| W_t \|_{op} \geq r \right)
  = 2 \PP \left( \sqrt{\frac{t n}{2}} \| \Gamma \|_{op} \geq r \right), \end{align*}
 where $\Gamma$ is a Gaussian Orthogonal Ensemble (GOE) random matrix. This means that  $\Gamma = (\Gamma_{ij})_{i,j=1,\ldots,n} \in \RR^{n \times n}_{\mathrm{sym}}$ is a random symmetric matrix such that $(\Gamma_{ij})_{i \leq j}$ are independent, centered Gaussian random variables, with $\EE \Gamma_{ij}^2 = (1 + \delta_{ij})/n$.
 It is well-known and proven by an epsilon-net argument (e.g., Vershynin \cite[Corollary 4.4.8]{vershynin}) that for $s \geq 1$,
 $$ \PP \left( \| \Gamma \|_{op} \geq C s \right) \leq 4 \exp(-s^2 n) $$
for a universal constant $C > 0$. The corollary is proven by setting $s = r / \sqrt{tn}$.
\end{proof}

Corollary \ref{cor_1833} implies that if $t < c/n$ and $a_0 \geq 1/2$, then the ellipsoid $\cE_t$ is typically
sandwiched between two concentric Euclidean balls whose radii $r_1 < r_2$ satisfy $r_2 / r_1 \leq C$.
Our next goal is to study the volume growth of the ellipsoid $\cE_t$, or equivalently, the decay of the determinant of the positive-definite matrix $A_t$. To this end we consider the non-negative, integer-valued random variable
\begin{equation}  N_t = \dim(F_t) \qquad \qquad \qquad (t \geq 0) \label{eq_2227_} \end{equation}
where $F_t = F_{A_t}$ is defined in (\ref{eq_1730}) and in Proposition \ref{prop_1510}.

\begin{lemma} For any fixed $T > 0$,
$$ \EE \log \det A_T \leq n \log a_0 -\frac{1}{2} \int_0^T \EE \left[ \| A_t \|_{op}^{-2} \cdot N_t \right] dt. $$
\label{lem_1239}
\end{lemma}

\begin{proof} For two fixed matrices $P, B \in \RR^{n \times n}_{\mathrm{sym}}$ with $P$ being positive-definite,
we have the Taylor expansion  as $\eps \to 0$,
$$ \log \det(P + \eps B) = \log \det P + \eps \Tr[P^{-1} B] - \frac{\eps^2}{2} \Tr[ (P^{-1} B)^2 ]+ O(\eps^3). $$
Denote the eigenvalues of $P$, repeated according to their multiplicity, by $\lambda_1,\ldots, \lambda_n \in (0, \infty)$.
Then for any orthonormal basis of eigenvectors $u_1,\ldots,u_n \in \RR^n$ corresponding to these eigenvalues,
$$ \left. \frac{d^2}{d \eps^2} \log \det(P + \eps B) \right|_{\eps = 0} = -\sum_{i,j=1}^n \frac{ (B u_i \cdot u_j)^2}{\lambda_i \lambda_j}
= -\sum_{i,j=1}^n \frac{ \langle B,  u_i \otimes_s  u_j\rangle^2}{\lambda_i \lambda_j}, $$
where $x \otimes_s y = (x \otimes y + y \otimes x)/2$
for $x,y \in \RR^n$. Observe that $(u_i \otimes_s u_j)_{i \leq j}$ is an orthogonal basis
of $\RR^{n \times n}_{\mathrm{sym}}$ with $| u_i \otimes_s u_j |^2 = (1 + \delta_{ij})/2$. Therefore,  for any linear map $S: \RR^{n \times n}_{\mathrm{sym}} \rightarrow \RR^{n \times n}_{\mathrm{sym}}$,
$$ \Tr[S] = \sum_{i,j=1}^n \left \langle S ( u_i \otimes_{s} u_j), u_i \otimes_{s} u_j \right \rangle. $$
We thus have  convenient representations for the first and second derivatives of the
function $F(X) = \log \det X$, which is a smooth function in the cone of symmetric, positive-definite matrices.
Our next step is to apply the It\^o formula (e.g. \cite[Section 5.2]{legall}) for the martingale $(A_t)_{t \geq 0}$
and the function $F$. Recall from Corollary \ref{cor_1852} that
$$ d A_t = \pi_t(dW_t), $$
and that $\pi_t: \RR^{n \times n}_{\mathrm{sym}} \rightarrow \RR^{n \times n}_{\mathrm{sym}}$ is an orthogonal projection.
Almost surely, for all $t > 0$ the matrix $A_t$ is positive definite. Hence the remark in \cite[Section 5.2]{legall}
applies and we obtain that almost surely, for $t > 0$
\begin{equation} \log \det A_t  = \log \det A_0 + \int_0^t \langle A_s^{-1}, \pi_s(d W_s) \rangle - \frac{1}{2} \int_0^t \delta_s ds,
  \label{eq_1138}
\end{equation}
 where
 $$ \delta_t = \sum_{i,j=1}^n \frac{ |\pi_t( u_i \otimes_{s} u_j)|^2}{\lambda_i \lambda_j}, $$
 and  $\lambda_1,\ldots,\lambda_n \in (0, \infty)$ are the eigenvalues of $A_t$
 while $u_1,\ldots,u_n \in \RR^n$ constitute a corresponding orthonormal basis of eigenvectors.
 We would like to obtain a lower bound on the expectation of the expression in (\ref{eq_1138}).
 As for the last summand in (\ref{eq_1138}), since $\lambda_i \leq \| A_t \|_{op}$ for all~$i$, we have
  \begin{equation}  \delta_t \geq \frac{1}{\| A_t \|_{op}^2} \sum_{i,j=1}^n |\pi_t( u_i \otimes_{s} u_j)|^2
 = \frac{1}{\| A_t \|_{op}^2} \sum_{i,j=1}^n \left \langle \pi_t ( u_i \otimes_{s} u_j), u_i \otimes_{s} u_j \right \rangle
 = \frac{1}{\| A_t \|_{op}^2} \cdot \Tr[\pi_t]. \label{eq_2216} \end{equation}
 We would like to show that the stochastic integral in (\ref{eq_1138}) has zero expectation.
 To this end, we will show  that the local martingale
 \begin{equation}  M_t = \int_0^t \langle A_s^{-1}, \pi_s(d W_s) \rangle = \int_0^t \langle \pi_s(A_s^{-1}), d W_s \rangle \qquad \qquad \qquad (t \geq 0) \label{eq_1429} \end{equation}
is in fact a martingale. Indeed, for $t > 0$,
$$ \EE |\pi_t(A_t^{-1})|^2 \leq  \EE |A_t^{-1}|^2 \leq n \cdot \EE \| A_t^{-1} \|_{op}^2  \leq n \cdot \EE \frac{\| A_t \|_{op}^{2(n-1)}}{\det^2 A_t} \leq \frac{n}{c_L^2} \cdot \EE \| A_t \|_{op}^{2(n-1)},  $$
where we used (\ref{eq_1238}) in the last passage. By Corollary \ref{cor_1833}, for any fixed $t > 0$, the random variable $\| A_t \|_{op}$
has a sub-gaussian tail, and hence, $$ \EE \| A_t \|_{op}^{2(n-1)} \leq C_n (a_0 + \sqrt{t})^{2(n-1)} $$
for some constant $C_n$ depending only on $n$. It follows that for $t > 0$,
$$ \EE \int_0^t |\pi_s(A_s^{-1})|^2 ds < \infty. $$
This implies that  $(M_t)_{t \geq 0}$ as defined in (\ref{eq_1429}) is indeed a martingale, see \cite[Section 5.1.1]{legall} or \cite[Chapter 3]{O}.
 Since $ \EE M_T = M_0 = 0$ and $A_0 = a_0 \cdot \id$,  it follows from (\ref{eq_1138}) and (\ref{eq_2216}) that
$$ \EE \log \det A_T = \log \det A_0 - \frac{1}{2} \int_0^T \EE \delta_t dt \leq n \log a_0 - \frac{1}{2} \int_0^T \EE
\| A_t \|_{op}^{-2} \cdot \Tr[\pi_t] dt. $$
 Since $\pi_t$ is the orthogonal projection operator onto the subspace $F_t \subseteq \RR^{n \times n}_{\mathrm{sym}}$ we have $\Tr[\pi_t] = \dim(F_t) = N_t$,
 and the lemma is proven.
\end{proof}

Recall the $L$-free evolving ellipsoid $\cE_t = \cE_{A_t}$ from Proposition \ref{prop_1510}.

\begin{proposition} Fix $0 < T \leq 20 \cdot n^{-5/3}$ and assume that $1 \leq a_0 \leq 1+ 10/n$. Then,
$$ \EE \log \det A_T \leq C -  \frac{n^2 T}{4} + \frac{1}{4} \int_0^T \EE |\partial \cE_t \cap L| dt, $$
where $C > 0$ is a universal constant. \label{cor_1945}
\end{proposition}

\begin{proof} Fix $0 < t \leq T$ and let $\cS_t$ be the event that $$ \| A_t - a_0 \cdot \id \|_{op} \leq C_0 \sqrt{ tn}, $$
where $C_0 > 0$ is the constant from Corollary \ref{cor_1833}. Let $1_{\cS_t}$ be the indicator of $\cS_t$, that equals $1$ if the event
$\cS_t$ occurs, and that vanishes otherwise. Then,
\begin{align} \nonumber \EE \left[ \| A_t \|_{op}^{-2} \cdot N_t \right] & \geq \EE \left[ 1_{\cS_t} \| A_t \|_{op}^{-2} \cdot N_t \right] \geq (a_0 + C \sqrt{t n})^{-2} \EE [1_{\cS_t} N_t ] \\ & = (a_0 + C \sqrt{t n})^{-2} \left( \EE[N_t] - \EE[(1 - 1_{\cS_t}) N_t] \right). \label{eq_2227}
\end{align}
Since $N_t = \dim(F_t) \leq n (n+1) / 2 \leq n^2$, by Corollary \ref{cor_1833},
$$ \EE[(1 - 1_{\cS_t}) N_t]  \leq n^2 \EE[1 - 1_{\cS_t}] = n^2 \cdot \PP( \| A_t - a_0 \cdot \id \|_{op} > C_0 \sqrt{ tn} ) \leq C n^2 e^{-n}. $$
Therefore, by using (\ref{eq_2227}) and the inequalities $|a_0 -1| \leq C /n$ and $N_t \leq n^2$,
$$ \EE \left[ \| A_t \|_{op}^{-2} \cdot N_t \right]
\geq (1 - C' \sqrt{tn} - C/n) \EE[N_t] - \tilde{C} e^{-n/2}
\geq \EE[N_t] - C' \sqrt{t} \cdot n^{5/2} - \bar{C}  n.
$$
By integrating over $t$ and recalling that $T \leq C n^{-5/3}$ we thus obtain
\begin{align*} \int_0^T \EE \left[ \| A_t \|_{op}^{-2} \cdot N_t \right]  dt
 \geq \int_0^T \EE \left[ N_t \right] dt - C' T^{3/2}n^{5/2} - T \cdot \bar{C} n
  \geq \int_0^T \EE \left[ N_t \right] dt - \hat{C}.
\end{align*}
Therefore, from Lemma \ref{lem_1239},
$$ \EE \log \det A_T \leq n \log (1 + 10/n) -\frac{1}{2} \int_0^T \EE \left[ \| A_t \|_{op}^{-2} \cdot N_t \right] dt \leq C' - \frac{1}{2}
\int_0^T \EE \left[ N_t \right] dt. $$
The subspace $F_t = F_{A_t}$ is defined in (\ref{eq_1730}) as the orthogonal complement in $\RR^{n \times n}_{\mathrm{sym}}$
to the subspace~$E$ spanned by $x \otimes x \ \ (x \in \partial \cE_t \cap L)$. The dimension of the subspace $E$
is at most $|\partial \cE_t \cap L| / 2$, since $\partial \cE_t \cap L = -(\partial \cE_t \cap L)$ while $0 \not \in \partial \cE_t \cap L$. Therefore,
$$ N_t = \dim(F_t) = \dim(\RR^{n \times n}_{\mathrm{sym}}) - \dim(E) \geq \frac{n(n+1)}{2} - \frac{|\partial \cE_t \cap L|}{2}. $$
Consequently,
$$ \EE \log \det A_T \leq C' - \frac{1}{2} \int_0^T \EE \left[ N_t \right] dt
\leq C' - T \frac{n(n+1)}{4} + \frac{1}{4} \int_0^T \EE |\partial \cE_t \cap L| dt , $$
completing the proof.
\end{proof}

\begin{remark} \label{rem_1814}
In the notation of Proposition \ref{prop_1510}, for $t = \tau_* = \tau_M$ the $L$-free ellipsoid $\cE_t \subset \RR^n$ almost surely satisfies
\begin{equation}  |\partial \cE_t \cap L| \geq n(n+1). \label{eq_1813} \end{equation}
Indeed, since $F_t = \{ 0 \}$ by Proposition \ref{prop_1510}(B),
we know that the matrices $x \otimes x \ \ (x \in \partial \cE_t \cap L)$
span $\RR^{n \times n}_{\mathrm{sym}}$. The number of such distinct matrices is at most $|\partial \cE_t \cap L|/2$,
since $\partial \cE_t \cap L = -(\partial \cE_t \cap L)$.
Thus (\ref{eq_1813}) follows from the fact that $\dim(\RR^{n \times n}_{\mathrm{sym}}) = n(n+1)/2$.
\end{remark}

\section{Lattice points on the boundary of the evolving ellipsoid}
\label{sec4}

We keep the notation and assumptions of the previous section. Thus $L \subset \RR^n$ is a fixed lattice, $a_0 > 0$ is such that the matrix $ a_0 \cdot \id$
is $L$-free, and we study the stochastic process $(A_t)_{t \geq 0}$ introduced in Proposition \ref{prop_1510}.
The following proposition is a step toward showing that, for a typical lattice $L \subset \RR^n$,
the number of lattice contact points of the $L$-free ellipsoid $\cE_t = \cE_{A_t}$ is usually at most $$ C \exp(n^2 t/8). $$

\begin{proposition} For any fixed $T > 0$ and $0 \neq x \in L$,
$$ \PP( x \in \partial \cE_T ) \, \leq \, 2 \PP \left( Z  \geq \frac{1}{\sqrt{T}} \left( a_0 - \frac{1}{|x|^2} \right) \right), $$
where $Z \sim N(0,1)$ is a standard Gaussian random variable.
\label{lem_1503}
\end{proposition}

\begin{proof} Since the matrix $a_0 \cdot \id$ is $L$-free, necessarily $a_0 |x|^2 \geq 1$. If $a_0 |x|^2 = 1$ then the conclusion of the lemma holds trivially, so let us assume that $a_0 |x|^2 > 1$.
	For $t \geq 0$ denote $$ M_t = A_t x \cdot x  - 1 =  \langle A_t, x \otimes x \rangle - 1.$$ If $M_t > 0$ then $x \not \in \partial \cE_t$, and hence
\begin{equation}
\PP( x \in \partial \cE_T ) \, \leq \, \PP(M_T \leq 0) \, \leq \, \PP \left( \inf_{0 \leq t \leq T} M_t \leq 0
\right). \label{eq_1438}
\end{equation}
By Corollary \ref{cor_1852},
	$$ d M_t =  \langle \pi_t (dW_t), x \otimes x \rangle =
	 \langle dW_t, \pi_t (x \otimes x) \rangle. $$
Thus $(M_t)_{t \geq 0}$ is a continuous martingale with $M_0 = a_0 |x|^2 - 1 > 0$. Its quadratic variation is given by
\begin{equation}  [M]_t =  \int_0^t |\pi_s(x \otimes x)|^2 ds \leq  \int_0^t |x \otimes x|^2 ds = t |x|^4, \label{eq_1430} \end{equation}
where we used the fact that $\pi_s$ is an orthogonal projection.
As in \cite[Section 5.3.2]{legall}, for $t \geq 0$ denote
\begin{equation}
R_t = \inf \{ s \geq 0 \, ; \, [M]_s \geq	 t \}, \label{eq_1842} \end{equation}
where the infimum of an empty set is defined as $+\infty$.
Almost surely, the function $R_t$ is non-decreasing in $t$
and  $M_{R_{t}}$ is continuous in $t$ in the set $\{ t \geq 0 \, ; \, R_t < \infty \}$.  The Dambis-Dubins-Schwartz Theorem (e.g. \cite[Theorem 5.13]{legall} or \cite[Chapter V]{RY}) states that
there exists a standard Brownian motion $(B_t)_{t \geq 0}$ in $\RR$
such that for all $t \geq 0$,
$$ M_{R_t} - M_0 = B_t $$
whenever $R_t < \infty$.
It follows from (\ref{eq_1430}) and (\ref{eq_1842}) that $R_t \geq t / |x|^4$.  Consequently, almost surely,
$$ \inf_{0 \leq t \leq T |x|^4} [M_0 + B_t] \leq
\inf_{0 \leq t \leq T |x|^4 \atop{R_t < \infty}} M_{R_t}
\leq \inf_{0 \leq t \leq T} M_t,  $$
where we used the continuity of $M_{R_t}$ in the last passage. Thus, from (\ref{eq_1438}),
$$
\PP( x \in \partial \cE_T )  \leq \PP \left( \inf_{0 \leq t \leq T |x|^4} [M_0 + B_t] \leq 0\right). $$
By the reflection principle for the standard Brownian motion (e.g.  \cite[Section III.3]{RY}),
$$   \PP \left (\inf_{0 \leq t \leq T|x|^4} [M_0 + B_t] \leq 0 \right)
= \PP \left (\sup_{0 \leq t \leq T|x|^4} B_t \geq M_0  \right)
= 2 \PP \left (B_{T |x|^4} \geq M_0 \right). $$
The law of $B_{T |x|^4}$ is the same as the law of $\sqrt{T} |x|^2 \cdot Z$. Therefore,
$$ \PP( x \in \partial \cE_T ) \leq 2 \PP \left( \sqrt{T} |x|^2 \cdot Z  \geq M_0
\right)=  2 \PP \left( Z  \geq \frac{1}{\sqrt{T}} \left( a_0 - \frac{1}{|x|^2} \right) \right). $$
\end{proof}

For $r \geq 0$ denote
\begin{equation}
 \Phi(r) = \min \left \{ \frac{1}{2}, \frac{e^{-r^2/2}}{\sqrt{2 \pi} \cdot r} \right \},
\label{eq_2237} \end{equation}
with $\Phi(0) = \min \{1/2, +\infty \} = 1/2$.
It is well-known that if $Z$ is a standard Gaussian random variable, then  for $r \geq 0$,
\begin{equation} \PP(Z \geq r) = \frac{1}{\sqrt{2 \pi}} \int_r^{\infty} e^{-x^2/2} dx \leq \min \left \{ \frac{1}{2}, \frac{1}{\sqrt{2 \pi}} \int_r^{\infty} \frac{x}{r} \cdot e^{-x^2/2} dx \right \} = \Phi(r). \label{eq_2226}
\end{equation}
For $t \geq 0$ we define
\begin{equation}  D_t = \{ x \in \RR^n \, ; \, (a_0 - C_0 \sqrt{tn}) |x|^2 < 1 \}, \label{eq_1212} \end{equation} where
$C_0 > 0$ is the universal constant from
Corollary \ref{cor_1833}. For $t > 0$ we consider the non-negative number
\begin{equation}  K_t(L) = \sum_{0 \neq x \in L  \cap D_t} \Phi \left( \frac{1}{\sqrt{t}} \left( a_0 - \frac{1}{|x|^2}  \right) \right).  \label{eq_1829_} \end{equation}

\begin{proposition} For any $t > 0$,
$$  \EE |\partial \cE_t \cap L| \leq 2 K_t(L) + C e^{-c n}, $$
where $C, c > 0$ are universal constants.
\label{cor_1832}
\end{proposition}

\begin{proof}	By the linearity of expectation and Proposition
	\ref{lem_1503},
	$$ \EE |\partial \cE_t \cap L \cap D_t| = \sum_{0 \neq x \in L \cap D_t} \PP(x \in \partial \cE_t) \leq 2\sum_{0 \neq x \in L \cap D_t} \PP \left( Z  \geq \frac{1}{\sqrt{t}} \left( a_0 - \frac{1}{|x|^2} \right) \right), $$
	where $Z$ is a standard Gaussian random variable.
	The matrix $a_0 \cdot \id$ is $L$-free. Thus $a_0 - 1/|x|^2 \geq 0$ for $x \in L$, and we may use the  bound (\ref{eq_2226})
	and the definition (\ref{eq_1829_}) of $K_t(L)$ to conclude that
\begin{equation} \EE |\partial \cE_t \cap L \cap D_t| \leq  2K_t(L).
\label{eq_658} \end{equation}
We claim that
\begin{equation}
\PP \left( \partial \cE_t \cap L \cap D_t \neq \partial \cE_t \cap L \right) \leq C e^{-n}.
\label{eq_659} \end{equation}
Indeed, it suffices to  prove that
$$
	\PP \left( \overline{\cE_t} \subseteq D_t \right) \geq 1 - C e^{-n},
$$
where $\overline{\cE_t} \subset \RR^n$ is the closure of the ellipsoid $\cE_t = \cE_{A_t} \subset \RR^n$. Equivalently, we need to show that
$$ \PP \left( A_t > \left( a_0 - C_0 \sqrt{tn} \right) \id \right) \geq 1 - C e^{-n}. $$
This follows from Corollary \ref{cor_1833}, proving (\ref{eq_659}). Recall from (\ref{eq_643})
that $|\partial \cE_t \cap L| \leq 2 \cdot (2^n - 1)$ as $\cE_t$ is an $L$-free ellipsoid and $L \subset \RR^n$ is a lattice. Thus, from (\ref{eq_658}) and (\ref{eq_659}),
\begin{align*}
\EE |\partial \cE_t \cap L| & \leq \PP \left( \partial \cE_t \cap L \cap D_t \neq \partial \cE_t \cap L \right) \cdot 2 \cdot (2^n - 1) + \EE |\partial \cE_t \cap L \cap D_t|
\\ & \leq C (2/e)^n + 2K_t(L),
\end{align*}
completing the proof.
\end{proof}

In view of Proposition \ref{cor_1945} and Proposition \ref{cor_1832}, we would like to understand
how large $K_t(L)$ is for a typical lattice $L$. Recall that  the parameter $K_t(L)$ is the {\it sum} of the function
\begin{equation} x \mapsto  \Phi \left( \frac{1}{\sqrt{t}} \left( a_0 - \frac{1}{|x|^2}  \right) \right) \label{eq_1708_} \end{equation}
over all non-zero lattice points in a
certain Euclidean ball. In the next  lemma we analyze the {\it integral} of the function from (\ref{eq_1708_}) over
a spherical shell approximating this ball.
This lemma will subsequently be combined with the Siegel summation formula, which states that for random lattices, the expectation of such a sum equals the corresponding integral.
Recall that $\Vol_n(B^n)$ is the volume of the $n$-dimensional unit ball.

\begin{lemma} Let $t > 0$. Assume that $0 < t \leq 20 n^{-2} \cdot \log n$
and $1 \leq a_0 \leq 1 + 10/n$. Consider the spherical shell
\begin{equation}  R = R_t = \left \{ x \in \RR^n \, ; \,  \frac{1}{a_0} \leq |x|^2 < \frac{1}{a_0 - C_0 \sqrt{tn}} \right \},
\label{eq_1210} \end{equation}
 where
	$C_0 > 0$ is the universal constant from
	Corollary \ref{cor_1833}. Then,
	$$ \int_R \Phi \left( \frac{1}{\sqrt{t}} \left( a_0 - \frac{1}{|x|^2}  \right) \right) dx \leq C e^{n^2 t/8} \cdot \Vol_n(B^n), $$
where $C > 0$ is a universal constant.
	\label{lem_1102}
\end{lemma}

\begin{proof} Denote $\kappa_n = \Vol_n(B^n)$ and $a_1 = a_0 - C_0 \sqrt{tn}$. Integrating in polar coordinates,
\begin{align*}
I :=	\int_R  \Phi \left( \frac{1}{\sqrt{t}} \left( a_0 - \frac{1}{|x|^2}  \right) \right) & dx = n \kappa_n \int_{1/\sqrt{a_0}}^{1/\sqrt{a_1}} \Phi \left( \frac{ a_0 - 1/r^2}{\sqrt{t}} \right) r^{n-1} dr.
\end{align*}
Changing variables $y = t^{-1/2} (a_0 - 1/r^2)$  we see that
$$ I = \frac{n \kappa_n \sqrt{t}}{2 } \int_0^{C_0 \sqrt{n}}  \frac{\Phi(y)}{(a_0 - \sqrt{t} y)^{\frac{n+2}{2}}}  dy
=\frac{n \kappa_n \sqrt{t}}{2 } a_0^{-\frac{n+2}{2}} \int_0^{C_0 \sqrt{n}}  \frac{\Phi(y)}{(1 - y\sqrt{t}/a_0)^{\frac{n+2}{2}}}  dy.
$$
Recall that $a_0 \geq 1$ while $t \leq 20 n^{-2} \cdot \log n$. Thus $y\sqrt{t}/a_0 \leq y\sqrt{t} < 1$ for all $y \in (0, C_0 \sqrt{n})$,
assuming that $n$ exceeds a certain given universal constant. Consequently,
\begin{equation}  \frac{I}{\kappa_n} \leq \frac{n \sqrt{t}}{2}  \int_0^{C_0 \sqrt{n}}  \frac{\Phi(y)}{(1 - y\sqrt{t})^{\frac{n+2}{2}}}  dy =  \frac{n \sqrt{t}}{2}  \cdot \left( I_1 + I_2 + I_3 \right), \label{eq_1150} \end{equation}
where $I_1$ is the integral from $0$ to $1$, where $I_2$ is the integral from $1$ to $\log n$ and where $I_3$ is the integral from $\log n$ to $C_0 \sqrt{n}$.
Begin by bounding $I_1$. To this end we will use
the elementary inequality $1-x \geq \exp(-2x)$ for $0 < x \leq 1/2$.
Since $\Phi(y) \leq 1/2$ and $t \leq 20 n^{-2} \cdot \log n$,
\begin{equation}  I_1 = \int_0^{1} \frac{\Phi(y)}{(1 - y\sqrt{t})^{\frac{n+2}{2}}}  dy
\leq \frac{1}{2} (1 - \sqrt{t})^{-\frac{n+2}{2}} \leq e^{(n+2) \sqrt{t}}
\leq C e^{n \sqrt{t}} \leq C' \frac{e^{n^2 t / 8}}{n \sqrt{t}}, \label{eq_1152_}
\end{equation}
where we used  the bound $e^x \leq C x^{-1} \cdot e^{x^2 / 8}$ for $x > 0$,
as well as our  standing assumption that $n$ is sufficiently large.
Next, we bound $I_3$  using the same elementary inequality. Since  $\sqrt{t} \leq 5 n^{-1} \cdot \sqrt{\log n}$,
$$ I_3 \leq \int_{\log n}^{C_0 \sqrt{n}} \Phi(y) e^{ (n+2) y \sqrt{t}}  dy \leq \int_{\log n}^{\infty} e^{-y^2/2 + 2 C y \sqrt{\log n}} dy
= e^{2 C^2 \log n} \int_{\log n - 2 C \sqrt{\log n}}^{\infty} e^{-x^2/2} dx.
 $$
 The last integral is at most $C' e^{-c' \log^2 n}$
by a standard bound for the Gaussian tail
 such as
 (\ref{eq_2226}) above. Consequently,
 \begin{equation}  I_3 \leq C' e^{2 C^2 \log n - c' \log^2 n}  \leq \bar{C} \leq \hat{C} \frac{e^{n^2 t / 8}}{n \sqrt{t}}, \label{eq_1152} \end{equation}
 as $c \leq x^{-1} \cdot e^{x^2/8}$ for $x > 0$. For the estimation of the integral $I_2$  we use the elementary inequality
 $1-x \geq \exp(-x - x^2)$ for $0 < x < 1/2$,
 as well as the bound $\sqrt{t} \leq 5 n^{-1} \cdot \sqrt{\log n}$. This yields
\begin{align*} I_2 & =  \int_1^{\log n} \frac{\Phi(y)}{(1 - y\sqrt{t})^{\frac{n+2}{2}}}  dy
	\leq \int_1^{\log n} \Phi(y) e^{\frac{(n+2) \sqrt{t}}{2}  y + \frac{(n+2) y^2 t}{2}} dy \\ & \leq C' \int_1^{\log n} \frac{e^{-y^2/2}}{y}  e^{\frac{n \sqrt{t}}{2}  y} dy = C' e^{n^2 t / 8} \int_1^{\log n} \frac{e^{-(y - n \sqrt{t}/2)^2/2}}{y}  dy.	
\end{align*}
Therefore,
\begin{align} \nonumber
I_2	 &\leq C' e^{n^2 t / 8}  \left[  \left| \int_{1}^{n \sqrt{t} / 4} e^{-(y - n \sqrt{t} / 2)^2/2} dy  \right|
	+ \int_{n \sqrt{t} / 4}^{\infty} \frac{e^{-(y - n \sqrt{t} / 2)^2/2}}{y} dy  \right]
	\\ & \leq \bar{C} e^{n^2 t / 8}  \left[ \PP(Z \geq n \sqrt{t}/4) + \frac{4}{n \sqrt{t}} \cdot \sqrt{2 \pi} \right] \leq \tilde{C} \frac{e^{n^2 t / 8}}{n \sqrt{t}}, \label{eq_1151}
\end{align}
where $Z$ is a standard Gaussian random variable, and we used a standard tail estimate such as (\ref{eq_2226}) which gives
$\PP(Z \geq n \sqrt{t}/4) \leq C / (n \sqrt{t})$.
To summarize, by (\ref{eq_1150}), (\ref{eq_1152_}), (\ref{eq_1152})  and (\ref{eq_1151}),
$$ \frac{I}{\kappa_n} \leq \frac{n \sqrt{t}}{2}  \left[ I_1 + I_2 + I_3 \right] \leq
C e^{n^2 t / 8}, $$
completing the proof.
\end{proof}

\begin{remark} When $T = 16 n^{-2} \cdot \log n$, most of the contribution to the integral
in Lemma \ref{lem_1102} comes  from points $x \in \RR^n$ with
\begin{equation}  \left| \, |x| \, - \, \left(1 + 4 \frac{\log n}{n}\right) \, \right| \leq C \frac{\sqrt{\log n}}{n}, \label{eq_1526} \end{equation}
as can be seen from the proof. When $L \subset \RR^n$ is a random, uniformly distributed lattice as in the next section,
there will typically be about $n^4 e^{C \sqrt{\log n}}$ lattice points satisfying (\ref{eq_1526}).
\end{remark}

\section{Random lattices}
\label{sec5}

Write $\mathscr{X}_n$ for the space of all  lattices $L \subset \RR^n$ with
$$
	\Vol_n(\RR^n / L) = \Vol_n(B^n).
	$$
We emphasize that our normalization is {\it not} that of covolume one lattices, but rather we consider lattices whose covolume is the volume of the Euclidean unit ball.
The space $\mathscr{X}_n$ is a homogeneous space under the action of the group $SL_n(\RR) = \left \{ g \in \RR^{n \times n} \, ; \, \det(g) = 1 \right \}$, where the action of $g \in SL_n(\RR)$ on the lattice $L \subset \RR^n$
is the lattice $$ g.L = \{ g(x) \, ; \, x \in L \}. $$	
Minkowski and Siegel \cite{siegel} discovered that there is a unique Haar	{\it probability} measure on $\mathscr{X}_n$ which is invariant under the action of $SL_n(\RR)$.
When we say that $L \subset \RR^n$ is a random lattice distributed uniformly in $\mathscr{X}_n$, we refer to the Haar probability measure on $\mathscr{X}_n$.
For more information on random lattices we refer the reader e.g. to
Gruber and Lekkerkerker \cite[Section 19.3]{GL} or to Marklof \cite{marklof}.
Throughout this section we set
\begin{equation}  a_0 := (1 - 1/n)^{-2}. \label{eq_1206} \end{equation}
Clearly $1 \leq a_0 \leq 1 + 10/n$, as required in order to apply Proposition \ref{cor_1945} and Lemma \ref{lem_1102}.
Recall the parameter $K_t(L) \geq 0$ that is defined in (\ref{eq_1829_}) for any lattice $L \subset \RR^n$ and any time $t > 0$.

\begin{proposition}  Let $0 < T \leq 20 n^{-2} \cdot \log n$.
Then there exists a lattice $L \in \mathscr{X}_n$ such that $a_0 |x|^2 > 1$ for any $0 \neq x \in L$ and
\begin{equation}  \int_0^T K_t(L) dt \leq \frac{C}{n^2} \cdot e^{n^2 T / 8},
\label{eq_1307} \end{equation}
where $C > 0$ is a universal constant.
\label{cor_1259}
\end{proposition}

\begin{proof} Let $L \in \mathscr{X}_n$ be a random, uniformly distributed lattice.
The Siegel summation formula~\cite{siegel} states that for any measurable function $\vphi: \RR^n \rightarrow [0, +\infty)$,
\begin{equation}  \EE \sum_{0 \neq x \in L} \vphi(x) = \frac{1}{\Vol_n(B^n)} \cdot \int_{\RR^n} \vphi. \label{eq_1008} \end{equation}
  Write $f(x) = 1$ if $|x| \leq 1 - 1/n$ and $f(x) = 0$ otherwise.
By (\ref{eq_1008}),
$$  \EE \sum_{0 \neq x \in L} f(x) = \frac{1}{\Vol_n(B^n)} \cdot \int_{\RR^n} f = (1 - 1/n)^n \leq \frac{1}{e}. $$
Hence, by the Markov inequality,
\begin{equation} \PP \left( \exists 0 \neq x \in L \, ; \, |x| \leq 1 -1/n \right) \leq \frac{1}{e}. \label{eq_1208_} \end{equation}
Let $0 < t \leq T$, and let $R = R_t \subseteq \RR^n$ be the spherical shell defined in (\ref{eq_1210}). Denote
\begin{equation}  \tilde{K}_t(L) = \sum_{0 \neq x \in L  \cap R_t} \Phi \left( \frac{1}{\sqrt{t}} \left( a_0 - \frac{1}{|x|^2}  \right) \right), \label{eq_2145} \end{equation}
i.e., the difference between $\tilde{K}_t(L)$ and $K_t(L)$ is that we sum over the spherical shell $R_t$ rather than over the ball $D_t$.
According to (\ref{eq_1008}) and Lemma \ref{lem_1102},
\begin{align*} \EE \tilde{K}_t(L) & = \EE \sum_{0 \neq x \in L  \cap R_t} \Phi \left( \frac{1}{\sqrt{t}} \left( a_0 - \frac{1}{|x|^2}  \right) \right)
\\ & = \frac{1}{\Vol_n(B^n)} \cdot \int_{R_t} \Phi \left( \frac{1}{\sqrt{t}} \left( a_0 - \frac{1}{|x|^2}  \right) \right) dx \leq C_1 e^{n^2 t /8}.
\end{align*} Since $\tilde{K}_t(L)$ is non-negative, we may apply Fubini's theorem and conclude that
$$ \EE \int_0^T  \tilde{K}_t(L) dt \leq C_1 \int_0^T e^{n^2 t /8} dt \leq \frac{8 C_1}{n^2} \cdot e^{n^2 T/8}. $$
By the Markov inequality,
\begin{equation}  \PP \left( \int_0^T  \tilde{K}_t(L) dt \geq  \frac{16 C_1}{n^2} \cdot e^{n^2 T/8} \right) \leq \frac{1}{2}.
\label{eq_1214} \end{equation}
Since $1/2 + 1/e < 1$, we conclude from (\ref{eq_1208_}) and (\ref{eq_1214}) that there exists a lattice $L \in \mathscr{X}_n$
such that $|x| > 1 - 1/n$ for all $0 \neq x \in L$ and such that
\begin{equation}  \int_0^T  \tilde{K}_t(L) dt <  \frac{16 C_1}{n^2} \cdot e^{n^2 T/8}. \label{eq_1257} \end{equation}
From (\ref{eq_1206}) we thus see that  $a_0 |x|^2  > 1$
for any $0 \neq x \in L$. Therefore the matrix $a_0 \cdot \id$ is $L$-free, and from (\ref{eq_1212}) and (\ref{eq_1210}) we see that  $$ (L \setminus \{0 \}) \cap R_t = (L \setminus \{0 \}) \cap D_t
\qquad \qquad \text{for any} \ 0 < t \leq T. $$ Consequently, from (\ref{eq_1829_}) and (\ref{eq_2145}),
$$ \tilde{K}_t(L) = K_t(L) \qquad \qquad \qquad (0 < t \leq T). $$
The desired conclusion (\ref{eq_1307}) thus follows from (\ref{eq_1257}).
\end{proof}

Let $L \subset \RR^n$ be the lattice whose existence is guaranteed by Proposition \ref{cor_1259}. Thus the matrix $a_0 \cdot \id$
is $L$-free. We may therefore apply Proposition \ref{prop_1510},
and consider the stochastic process $$ (A_t)_{t \geq 0} $$ of positive-definite, symmetric $n \times n$ matrices. Recall that almost surely,
for any $t > 0$ the ellipsoid $\cE_t = \cE_{A_t}$ is $L$-free.

\begin{lemma} Set $T = 16 n^{-2} \cdot \log n$. Then with positive probability,
$$  \det A_{T} \leq \frac{C}{n^4}, $$
for a universal constant $C > 0$.
\label{lem_1305}
\end{lemma}

\begin{proof} From Proposition \ref{cor_1832} and Proposition \ref{cor_1259}, for any $t > 0$,
$$  \int_0^T \EE |\partial \cE_t \cap L| dt \leq \int_0^T \left( 2 K_t(L) + C' e^{-c' n} \right) dt
\leq \frac{2C}{n^2} \cdot e^{n^2 T / 8} + C' T e^{-c' n} \leq \tilde{C}, $$
since $n^2 T / 8 = 2 \log n$. By our standing assumption that $n$ is sufficiently large, we have
$T \leq 20 \cdot n^{-5/3}$. We may therefore apply Proposition \ref{cor_1945}, and conclude that
$$ \EE \log \det A_T \leq C -  \frac{n^2 T}{4} + \frac{1}{4} \int_0^T \EE |\partial \cE_t \cap L| dt
\leq C' -  \frac{n^2 T}{4} = C' - 4 \log n. $$
In particular, with positive probability, $\log \det A_T \leq C' - 4 \log n$. The lemma follows by exponentiation.
\end{proof}

 \begin{proof}[Proof of Theorem \ref{thm1}]
 Since the matrix $A_T$ is almost surely $L$-free,
 Lemma \ref{lem_1305} guarantees the existence of an $L$-free matrix $A \in \RR^{n \times n}_+$ with $$ \det(A) \leq C / n^4. $$
According to  (\ref{eq_1128}),
\begin{equation} \Vol_n(\cE_A) = \det(A)^{-1/2} \cdot \Vol_n(B^n) \geq c_0 n^2 \cdot \Vol_n(B^n). \label{eq_1945_} \end{equation}
 The ellipsoid $\cE_A$ is $L$-free, thus $L \cap \cE_A = \{ 0 \}$.
 All that remains is to normalize. Write $\kappa_n = \Vol_n(B^n)$ and consider the matrix
  $$ S = \kappa_n^{-1/n} \cdot \det(A)^{-1/(2n)} \cdot \sqrt{A}, $$
where $\sqrt{A} \in \RR^{n \times n}_+$  is  the positive-definite square root of the matrix $A \in \RR^{n \times n}_+$.
Note that $\cE_A = (\sqrt{A})^{-1}(B^n)$. Denote $$ \tilde{L} = S(L) \subset \RR^n. $$
The lattice $\tilde{L} \subset \RR^n$ has covolume one, since $L \in \mathscr{X}_n$.
By (\ref{eq_1945_}), the set $S(\cE_A)$ is a Euclidean ball centered at the origin of volume at least $c_0 n^2$.
If $K \subseteq \RR^n$ is an open Euclidean ball centered at the origin
with $\Vol_n(K) = c_0 n^2$, then $K \subseteq S(\cE_A)$ and hence $$ \tilde{L} \cap K \subseteq S(L \cap \cE_A) = \{ 0 \}. $$
 \end{proof}

\begin{remark} \label{rem_951} Our proof of Theorem \ref{thm1} suggests a randomized algorithm for  constructing the sphere-packing lattice $L \subset \RR^n$. Indeed, begin by sampling $L \in \mathscr{X}_n$ uniformly, for example by using Ajtai's algorithm \cite{ajtai}, and then run the stochastic process described above.
It would be interesting to carry out numerical simulations of our construction, or variants thereof (e.g., replacing $\pi_t$ in (\ref{eq_1701}) by another linear transformation with the same image -- perhaps the linear map $X \mapsto \pi_t(A_t^{\alpha} X)$ for a certain  exponent $\alpha$). Such numerical simulations could help  compute  the universal constant $c$ from Theorem \ref{thm1} yielded by this construction.
\end{remark}

\appendix
\section{Appendix}

\begin{lemma} Let $L \subset \RR^n$ be a lattice and let $\cE \subseteq \RR^n$ be a non-empty, open, origin-symmetric, bounded, strictly-convex set
(e.g., an origin-symmetric ellipsoid) with $\cE \cap L = \{ 0 \}$. Then,
$$ |\partial \cE \cap L| \leq 2 \cdot (2^n-1). $$
\end{lemma}

\begin{proof} We follow Minkowski's classical proof that the Voronoi cell of a lattice contains
at most $2 \cdot (2^n - 1)$ facets.  Since $\cE \cap L = \{ 0 \}$, no point of $\partial \cE$ can belong to $2 L$.
Moreover, we claim that for any $x, y  \in \partial \cE \cap L$ with $x \neq y$ and $x \neq -y$, necessarily
$x - y \not \in 2 L$. Indeed, otherwise $0 \neq x-y \in 2 L$ while  $$ \frac{x-y}{2} \in \left \{ \frac{x_1 + x_2}{2} \, ; \, x_1, x_2 \in \partial \cE, x_1 \neq x_2 \right \}  \subseteq \cE. $$ Thus $(x-y) /2$ is a non-zero point belonging both to $L$ and to $\cE$, in contradiction to $\cE \cap L = \{ 0 \}$.
Consequently each coset of the subgroup of $2 L$ of the lattice $L$, either contains no points from~$\partial \cE$,
or else contains a pair of antipodal points from $\partial \cE$. There are $2^n - 1$ such cosets, excluding the subgroup $2L$ itself
which contains no points from $\partial \cE$, and the union of these cosets covers $L \setminus (2L)$. Hence the cardinality of $\partial \cE \cap L$
is at most $2 \cdot (2^n - 1)$.
\end{proof}

Let us now restate Lemma 2.1 and provide a proof.

\medskip
\noindent
\textbf{Lemma 2.1.} \textit{
Let $M_t \in \RR^{n \times n}_{\mathrm{sym}} \ (t \geq 0)$ be a family of matrices depending continuously on $t \geq 0$, such that not all of the matrices are positive-definite. 	Assume that the matrix $M_0 \in \RR^{n \times n}_{\mathrm{sym}}$ is positive-definite and $L$-free, and that for  all $t \geq 0$,
\begin{equation*}
	\partial \cE_{M_0} \cap L \subseteq \partial \cE_{M_t} \cap L. \tag{\ref{eq_1935}} \end{equation*}
Then the following hold:
\begin{enumerate}
	\item[(A)] Denote $$ \tau := \sup \left \{ \, t \geq 0 \, ; \, M_s \ \textrm{is} \ L\textrm{-free with } \partial \cE_{M_s} \cap L = \partial \cE_{M_0} \cap L \textrm{ for all } s \in [0,t] \, \right \}. $$
Then $0 < \tau < \infty$.
	\item[(B)] The symmetric matrix $M_t$ is positive-definite and $L$-free for all $0 \leq t \leq \tau$.
	\item[(C)] We gained at least one additional contact point at time $\tau$. That is,
\begin{equation*}  \partial \cE_{M_0} \cap L \subsetneq \partial \cE_{M_\tau} \cap L. \tag{\ref{eq_1608}} \end{equation*}
\end{enumerate}
}
\medskip

\begin{proof} We claim that there exist $t_0, \eps > 0$ such that
for all $0 \leq t \leq t_0$ and $0 \neq x \in L \setminus \partial \cE_{M_0}$,
\begin{equation} M_t x \cdot x > 1 + \eps.
	\label{eq_1939} \end{equation}
In order to prove this claim, we use the fact that $M_0$ is positive-definite, and hence there exists $\eps_1 > 0$ such that $M_0 \geq \eps_1 \cdot \id$.
The symmetric matrix $M_t$ depends continuously on $t$, and hence for some $t_1 > 0$ we have $M_t \geq (\eps_1/2) \id$ for all $0 \leq t \leq t_1$. Therefore (\ref{eq_1939}) holds true for all $|x| > 2/ \sqrt{\eps_1}$, provided that $\eps < 1$ and $t_0 \leq t_1$. It remains  to prove (\ref{eq_1939}) for $x \in F$ where
\begin{equation}  F = \left \{ \, 0 \neq x \in L \setminus \partial \cE_{M_0} \, ; \, |x| \leq 2 / \sqrt{\eps_1} \, \right \}. \label{eq_1629} \end{equation}
The set $F$ is finite since $L$ is discrete. The set $F$ is disjoint from the  ellipsoid $\cE_{M_0}$ since $M_0$ is $L$-free. It thus follows from (\ref{eq_1629}) that $F$ is disjoint from the closure of the ellipsoid $\cE_{M_0}$,
and hence $M_0 x \cdot x > 1$ for all $x \in F$. Since $M_t$ depends continuously on $t$ while $F$ is finite, there exists $t_0 \in (0, t_1)$ and $\eps \in (0, 1)$ such that $M_t x \cdot x > 1 + \eps$ for all $x \in F$ and $0 \leq t \leq t_0$. This completes the proof of (\ref{eq_1939}).

\medskip Let us prove (A). Fix $0 \leq t \leq t_0$. It follows from (\ref{eq_1939})
that any point $0 \neq x \in L \setminus \partial \cE_{M_0}$ does not belong to $\partial \cE_{M_t}$, since  $M_t y \cdot y = 1$ for all $y \in \partial \cE_{M_t}$.
Hence $$ \partial \cE_{M_t} \cap (L \setminus \partial \cE_{M_0}) = \emptyset, $$ where
we also used the fact that $0 \not \in \partial \cE_{M_t}$. Consequently,
\begin{equation}  \partial \cE_{M_t} \cap L \subseteq   \partial \cE_{M_0} \cap L. \label{eq_1429_} \end{equation}
It follows from (\ref{eq_1935}) that the open set $\cE_{M_t}$
contains no points from $L \cap \partial \cE_{M_0}$.
It follows from (\ref{eq_1939}) that the set $\cE_{M_t}$
does not contain non-zero points from $L \setminus \partial \cE_{M_0}$.
Therefore $\cE_{M_t}$ does not contain
 non-zero points from $$ (L \cap \partial \cE_{M_0}) \cup (L \setminus \partial \cE_{M_0})  = L. $$
 In other words, the matrix $M_t$ is $L$-free.
It now follows from (\ref{eq_1935}), (\ref{eq_1429_}) and the definition of $\tau$ that
$$ \tau \geq t_0 > 0. $$
 Since $\cE_{M_t}$ is $L$-free for $0 \leq t < \tau$, by (\ref{eq_1358}),
 \begin{equation} \sup_{0 \leq t < \tau} \Vol_n( \cE_{M_t} ) \leq C_L < \infty. \label{eq_1129}
 \end{equation}
It follows from (\ref{eq_1128}) and (\ref{eq_1129}) that the matrix $M_t$ is positive-definite for all $0 \leq t < \tau$, and
\begin{equation} \inf_{0 \leq t < \tau} \det(M_t) > 0. \label{eq_1133} \end{equation} This implies in particular that $\tau < \infty$, since we assumed that  $(M_t)_{0 \leq t < \infty}$ is {\it not} a family of positive-definite matrices. Thus (A) is proven.

\medskip We move on to the proof of (B). We have seen that the matrix $M_t$  is $L$-free for $0 \leq t < \tau$, and hence the matrix $M_{\tau}$ is $L$-free as well, by continuity.
Since $M_t$ is positive-definite for $0 \leq t < \tau$, the matrix $M_{\tau}$ is positive semi-definite, by continuity. It follows from (\ref{eq_1133})
that $\det M_{\tau} > 0$ and hence $M_{\tau}$ is in fact positive-definite.
This completes the proof of (B).

\medskip We still need to prove (C). If (\ref{eq_1608}) does not hold true, then necessarily
\begin{equation} \partial \cE_{M_0} \cap L = \partial \cE_{M_\tau} \cap L, \label{eq_1221_} \end{equation} according  to (\ref{eq_1935}). Hence, by (\ref{eq_1935})
and (\ref{eq_1221_}),
\begin{equation}  \partial \cE_{M_\tau} \cap L \subseteq \partial \cE_{M_t} \cap L \qquad \qquad \qquad \text{for all} \ t \geq \tau.
\label{eq_1630} \end{equation}
The matrix $M_{\tau}$ is positive-definite and $L$-free according to (B).
Since $M_t$ is positive-definite for $0 \leq t \leq \tau$,
we know that $(M_{t+ \tau})_{t \geq 0}$ is a family of matrices depending continuously on $t$, not all of them positive-definite.
Therefore, thanks to (\ref{eq_1630}), we may apply the lemma for the family of matrices $(M_{t+ \tau})_{t \geq 0}$, and conclude from (A) that
$$ \tau_1 := \sup \left \{ \, t \geq \tau \, ; \, M_s \ \textrm{is} \ L\textrm{-free with } \partial \cE_{M_s} \cap L = \partial \cE_{M_\tau} \cap L \textrm{ for all } s \in [\tau,t] \, \right \} $$
satisfies $\tau_1 \in (\tau, \infty)$. However, equation (\ref{eq_1221_}) and the maximality property of $\tau$ implies that $\tau_1 = \tau$, in contradiction.
 \end{proof}

\medskip
\noindent School of Mathematical Sciences, Tel Aviv University, Tel Aviv 6997801, Israel; and
Department of Mathematics, Weizmann Institute of Science, Rehovot 7610001, Israel. \\
\textit{e-mail:} \href{mailto:klartagb@tau.ac.il}{\texttt{klartagb@tau.ac.il}}

\end{document}